\documentclass[10pt]{amsart}

\pdfoutput=1

\usepackage[]{tikz}
\usetikzlibrary{patterns}
\usetikzlibrary{3d,calc}
\usetikzlibrary{decorations.markings}
\usetikzlibrary{decorations.pathreplacing}

\usepackage{tikz-cd}

\usepackage{multicol}
\usepackage{amssymb}
\usepackage[10pt]{moresize}

\usepackage{array}

\usepackage{color}

\usepackage{hyperref}

\usepackage{enumitem}
\newlist{conditions}{enumerate}{1}
\setlist[conditions,1]{label=(\arabic*),ref=(\arabic*)}
\newlist{sublemmas}{enumerate}{1}
\setlist[sublemmas,1]{label=(\roman*),ref=\thetheorem~(\roman*)}

\setenumerate[1]{label=\thesection.\arabic*}
\setenumerate[2]{label*=.\roman*}

\makeatletter
\newtheorem*{rep@theorem}{\rep@title}
\newcommand{\newreptheorem}[2]{%
\newenvironment{rep#1}[1]{%
 \def\rep@title{#2 \ref{##1}}%
 \begin{rep@theorem}}%
 {\end{rep@theorem}}}
\makeatother

\newtheorem{theorem}{Theorem}[section]
\newreptheorem{theorem}{Theorem}
\newtheorem{lemma}[theorem]{Lemma}

\newtheorem{proposition}[theorem]{Proposition}
\newtheorem{corollary}[theorem]{Corollary}

\newtheorem{definition}[theorem]{Definition}

\theoremstyle{definition}\newtheorem{remark}[theorem]{Remark}\theoremstyle{plain}

\title[Topology and Dynamics of the Contracting Boundary]{Topology and Dynamics of the Contracting Boundary of Cocompact CAT(0) Spaces}

\author{Devin Murray}

\begin{document}
\pagestyle{myheadings}
\markright{Topology and Dynamics of CAT(0) Spaces\hfill}

\begin{abstract} Let $X$ be a proper CAT(0) space and let $G$ be a cocompact group of isometries of $X$ which acts properly discontinuously. Charney and Sultan constructed a quasi-isometry invariant boundary for proper CAT(0) spaces which they called the contracting boundary. The contracting boundary imitates the Gromov boundary for $\delta$-hyperbolic spaces. We will make this comparison more precise by establishing some well known results for the Gromov boundary in the case of the contracting boundary. We show that the dynamics on the contracting boundary is very similar to that of a $\delta$-hyperbolic group. In particular the action of $G$ on $\partial_cX$ is minimal if $G$ is not virtually cyclic. We also establish a uniform convergence result that is similar to the $\pi$-convergence of Papasoglu and Swenson and as a consequence we obtain a new North-South dynamics result on the contracting boundary. We additionally investigate the topological properties of the contracting boundary and we find necessary and sufficient conditions for $G$ to be $\delta$-hyperbolic. We prove that if the contracting boundary is compact, locally compact or metrizable, then $G$ is $\delta$-hyperbolic.

\end{abstract}

\maketitle

\section{Introduction}

The Gromov boundary has been a very useful and powerful tool in understanding the structure of $\delta$-hyperbolic groups. The boundary has a large array of nice topological, metric, and dynamical properties that can be used in probing everything from subgroups and splittings to algorithmic properties. It has also played an important role in proving various rigidity theorems.

For proper CAT(0) spaces there is a nice visual boundary, but Croke and Kleiner showed that such a boundary is not a quasi-isometry invariant \cite{CK}. They constructed two different CAT(0) spaces with non-homeomorphic visual boundaries on which the same group acts geometrically. The visual boundary can still be used to study CAT(0) groups, for instance it can detect products \cite[II.9.24]{BH}, but the failure of quasi-isometry invariance is a serious blow. In \cite{CS13} Charney and Sultan construct a natural topological space associated to a CAT(0) space, called the contracting boundary, which is a quasi-isometry invariant.

One of the many properties that geodesics have in a $\delta$-hyperbolic space is that there is a uniform bound on the diameter of the projection of a ball onto a geodesic disjoint from it. It turns out that this is a very powerful property and the existence of such geodesics in a space has significant consequences for the geometry \cite{A-K,BF09,H09}. Such geodesics are called contracting geodesics.

The contracting boundary, $\partial_cX$, of a complete CAT(0) space $X$ is the set of contracting rays in $X$ up to asymptotic equivalence. It is homeomorphic to the Gromov boundary when $X$ is also $\delta$-hyperbolic and is designed to imitate the Gromov boundary for more general CAT(0) spaces. However, the contracting boundary for CAT(0) groups is still not very well understood, so we hope to help lay out the ground work for a program of study to better understand it and its implications for CAT(0) groups.

The rank-rigidity conjecture of Ballman and Buyalo says that for sufficiently nice CAT(0) spaces, the non-existence of a periodic contracting axis implies that the space is either a metric product, a symmetric space, or a Euclidean building \cite{BB08}. Rank-rigidity theorems have been proven for many different classes of spaces including Hadamard manifolds, CAT(0) cube complexes, right angled Artin groups as well some others \cite{B85,CS11,BC12,CF10}. In light of these results, the study of CAT(0) groups can often be reduced to the study of CAT(0) groups with a contracting axis. Building off the work of Ballman and Buyalo we show that a CAT(0) group has a contracting axis if and only if the contracting boundary, $\partial_cX$, is non-empty. Thus the contracting boundary is a promising tool for the study of CAT(0) spaces and groups.

Several of the rigidity theorems for hyperbolic groups can be proven through a careful study of the dynamics of the action of the group on its boundary  \cite{F95,G92,CJ}. These rigidity theorems become even more striking when further geometric structures are added, such as the Mostow rigidity of finite dimensional hyperbolic manifolds \cite{M68}. 

Our most promising results have been predominantly dynamical. While the topology of the contracting boundary tends to be rather pathological, many of the dynamical properties of the Gromov boundary are shared by the contracting boundary. 

There are two main dynamics results that we obtain in this paper. The first says that the orbit of any contracting ray is dense. 

\begin{reptheorem}{thm:orbitsaredense} Let $G$ be a group acting geometrically on a proper CAT(0) space. Either $G$ is virtually $\mathbb{Z}$ or the $G$ orbit of every point in the contracting boundary is dense.

\end{reptheorem}

The second dynamics result concerns a more powerful convergence group like property. This is similar to the $\pi$-convergence of a CAT(0) group on its visual boundary. 

\begin{reptheorem}{thm:conv} Let $X$ be a proper CAT(0) space and $G$ a group acting geometrically on $X$. If $g_i$ is a sequence of elements of $G$ where $g_ix \to \gamma^+$ for some $x \in X$ and $\gamma^+ \in \partial_cX$, then there is a subsequence such that $g_i^{-1}x \to \gamma^-$ for $\gamma^- \in \partial_cX$ and for any open neighborhood $U$ of $\gamma^+$ in $\partial_cX$ and any compact $K$ in $\partial_cX - \gamma^-$ there is an $n$ such that $g_i(K) \subseteq U$ for $i \geq n$. 
\end{reptheorem}

The normal version of $\pi$-convergence introduced by Papasoglu and Swenson in \cite{PS} and the North-South dynamics due to Hamenst\"adt in \cite{H09} both deal with the visual topology on the visual boundary. Because the topology on the contracting boundary is not the subspace topology these theorems don't directly apply. 

Both of these results are well known for the action of a hyperbolic group on its boundary. We will discuss them both in greater detail in Section \ref{sec:dynamics}.

The topology on the contracting boundary is defined as a direct limit of subspaces, $\partial_c^DX$, consisting of rays with contracting constant bounded by $D$. The topology is quite fine and is, perhaps, more pathological than one would expect from a bordification. While the subspaces $\partial_c^DX$ are compact and metrizable, we prove that the direct limit, $\partial_cX$, is not always compact (nor locally compact) for CAT(0) groups, though it is known to be $\sigma$-compact \cite{CS13}. In Section \ref{sec:topology} we define the topology and discuss some of the basic topological facts about it.

One of the powerful tools that is available when studying the Gromov boundary is the family of metrics on it. In Section \ref{sec:metrizability} we show that a number of topological properties, including the metrizability of the contracting boundary, characterize when the space is $\delta$-hyperbolic. 

\begin{reptheorem}{thm:characterizationofhyperbolicity} Let $X$ be a complete proper CAT(0) space with a geometric group action. Then the following are equivalent:

\begin{sublemmas}

\item $X$ is $\delta$-hyperbolic.
\item $\partial_cX$ is compact.
\item $\partial_cX$ is locally compact.
\item $\partial_cX$ is metrizable.

\end{sublemmas}

\end{reptheorem}

A generalization of the contracting boundary for proper geodesic metric spaces, called the Morse boundary, was introduced by Cordes in \cite{Co15}. It would be interesting to see if any of these results hold true in that more general setting. In particular, it seems like many of the necessary pieces are already known for an analogue of Theorem~\ref{thm:characterizationofhyperbolicity} for the Morse boundary in some restricted cases \cite{Co15,F15}.

\section*{Acknowledgements}

The author would like to thank Ruth Charney and Matthew Cordes for always finding the time to answer the many questions that came up while working on this paper. The author would also like to thank the various referees for productive commentary which greatly improved this paper.

\section{Some basics on contracting geodesics}

For the entirety of this paper we will assume that a geodesic $a$ is an isometric embedding of $\mathbb{R}$ or a segement of $\mathbb{R}$ into a metric space. For convenience we will often conflate the image of the embedding with the map itself. For the subsequent discussion we may assume that unless stated otherwise all metric spaces are proper and satisfy the CAT(0) inequality.

\emph{Notation:} Recall that $\partial X$ is the set of infinite geodesics, where two geodesics are considered equivalent if they are within a bounded neighborhood of one another. Throughout we will mostly consider the cone topology on this set, recall that a neghborhood system is described by the set $U_a(\epsilon, r)$ which are all geodesics, $b$ starting at $a(0)$ such that $d(b,a(r)) < \epsilon $. This is sometimes called the cone topology. 

We will adopt the notation convention that $[x,y]$ represents the unique geodesic between the points $x,y \in X$. For a point $x \in X$ and a point $\alpha \in \partial X$ we will use half open intervals $[x,\alpha)$ to denote the unique geodesic starting from $x$ which is in the equivalence class $\alpha$. When we use double parentheses, $(\alpha,\beta)$, we will mean a specific bi-infinite geodesic, $c$, such that $c|_{(-\infty,0]} \in \alpha$ and $c|_{[0,\infty)} \in \beta$. For a geodesic ray $a$ in $X$ we will use $a(\infty)$ to denote the equivalence class of $a$ in $\partial X$.

\begin{definition}[Contracting geodesics] A geodesic $a$ is said to be \emph{$A$-contracting} for some constant $A$ if for all $x,y \in X$

$$d(x,y) < d(x,\pi_{a}(x)) \implies d(\pi_{a}(x),\pi_{a}(y)) < A$$

\end{definition}

Note that this definition is sometimes called \emph{strongly contracting} in the literature. Contracting geodesics can be thought of as detecting hyperbolic `directions' in a CAT(0) space. Another useful, and equivalent, property of hyperbolic like geodesics is that of $\delta$-slimness. This is much closer to the notion of Gromov hyperbolicity. 
\begin{definition}[Slim geodesics] A geodesic $a$ is said to be \emph{$\delta$-slim} if for all $y \notin a$ and all $z$ on $a$ there exists a point $w$ on the geodesic $[y,z]$ such that $d(\pi_{a}(y),w) \leq \delta$.

\end{definition}

It turns out that this property will be much more versatile for our purposes, luckily for us the two notions are equivalent in proper CAT(0) spaces. For a complete proof see \cite{CS13} and \cite{BF09}.

\begin{lemma} \label{lemma:slimgeodesics} If $a$ is a contracting geodesic with contracting constant $A$ then $a$ is $\delta_A$-slim for some $ \delta_A $ which depends only on $A$. The converse is also true, if $a$ is $\delta$ slim then it is $\Phi(\delta)$-contracting where $\Phi(\delta)$ depends linearly on $\delta$. \end{lemma}

I will adopt the convention that Bestvina--Fujiwara used in \cite{BF09}, that all constants will be denoted by $\Phi(\cdot)$. Typically the function $\Phi(\cdot)$ will be linear in its terms. When constants are referenced in later statements the relevant lemma and theorem number will be added as a subscript. Sometimes it will be expedient to drop the terms of $\Phi$ if they are clear from context, e.g. Lemma~\ref{lemma:slimgeodesics} says that if $a$ is a $\delta$-slim geodesic it is $\Phi_{\ref{lemma:slimgeodesics}}$-contracting.

One of the most important facts about the contracting constant of a geodesic is that it is controlled by the contracting constants of nearby geodesics. This will very important in the sequel as it will allow us to push contracting geodesics around via isometries and give us fine tuned control on the contracting constants of a target geodesic.

\begin{lemma} \label{lemma:closetocontracting} If we have two geodesics $[a,b]$ and $[a',b']$, where $[a,b]$ is $A$-contracting, $d(a,a') = D$, and $d(b,b') = D'$ then $[a',b']$ is $\Phi(A,D,D')$-contracting. It suffices to take $\Phi(A,D,D') = 16A + 28D + 7D' + 10$.  

\end{lemma}

A proof for this is in \cite{BF09} Lemma~3.8. Though they do not  write down the explicit $\Phi(A,D,D')$ in their paper it is possible to recover the one above from their work.

Another important property of contracting geodesics is that subsegments of a contracting geodesic are contracting. So unless otherwise specified we may assume that if $a$ is $A$-contracting all subsegments are also $A$-contracting.
  
\begin{lemma} \label{lem:subsegments} If $a$ is a contracting ray with contracting constant $A$ then a subsegment of it is $\Phi(A)$-contracting where $\Phi(A) = A + 3$.

\end{lemma}
\noindent Bestvina-Fujiwara prove this in \cite{BF09} in a slightly more general context. To understand why the contracting constant may have to increase, note that there are balls that don't intersect the subsegment but do intersect the original contracting ray. The increase can be thought of as making up for some possible differences in the local geometry of the subsegment compared to the original ray. 

\smallskip

As a converse to the previous lemma, sometimes we will need to piece together two contracting geodesics into a longer geodesic. It is an easy warm up exercise to show that this new geodesic is also contracting.

\begin{lemma} \label{lemma:concatenationsarecontracting} Let $a$ and $b$ be geodesics in a CAT(0) space $X$. If $a$ is $A$-contracting, $b$ is $B$-contracting, and $a(0) = b(0) = z $ then the following hold:

\begin{sublemmas}

\item If the concatenation of $a$ and $b$ is a geodesic, then it is $(A+B)$-contracting.

\item For every point $x \in a$ and $y \in b$, the geodesic $[x,y]$ is a $\Phi(A,B)$-contracting geodesic where $\Phi(A,B) = \Phi_{\ref{lemma:closetocontracting}}(A,0,\delta_A)+\Phi_{\ref{lemma:closetocontracting}}(B,\delta_A,0)$ is sufficient. 

\item If $X$ is also a proper metric space then the geodesic $[x, b(\infty))$ and $(a(\infty), b(\infty))$ are $\Phi(A,B)$-contracting such that $\Phi(A,B)$ is as above.

\end{sublemmas}
\end{lemma}

\begin{proof} \emph{(i)} Left as an exercise. \\

\noindent \emph{(ii)} If the concatenation of $a$ and $b$ is a geodesic this is obvious by part (i), so assume otherwise. By Lemma~\ref{lemma:slimgeodesics} we have that $a$ is $\delta_A$ slim and $b$ is $\delta_B$ slim. We may assume that $\delta_A \geq \delta_B$. By the definition of slimness there is a point $w$ on $[x,y]$ which is within $2\delta_A$ of both of the other sides of the geodesic triangle $\Delta(x,y,z)$. By Lemma~\ref{lemma:closetocontracting} the geodesic $[x,w]$ is $\Phi_{\ref{lemma:closetocontracting}}(A,0,2\delta_A)$-contracting and the geodesic $[w,y]$ is $\Phi_{\ref{lemma:closetocontracting}}(B,2\delta_A,0)$-contracting. So by the first part of this lemma $[x,y]$ is contracting with $\Phi(A,B) = \Phi_{\ref{lemma:closetocontracting}}(A,0,2\delta_A)+\Phi_{\ref{lemma:closetocontracting}}(B,2\delta_A,0)$. \\

\noindent\emph{(iii)} This follows easily from part \emph{(ii)} by taking a sequence of points $y_i \to b(\infty)$. The uniqueness of infinite rays in proper CAT(0) spaces, Lemma~\ref{lemma:closetocontracting}, and Arzel\'a--Ascoli implies the statement.

\end{proof}

In contrast with geodesics in Euclidean flats the diameter of the projection of any geodesic onto a contracting geodesic is finite. The proof can be found in \cite{CS13}.
 
 \begin{lemma} \label{lemma:boundedprojection}If $a$ is a contracting geodesic and $b$ is any other infinite geodesic then the projection of $b$ onto $a$ is of bounded diameter $D$. 
 
 \end{lemma}

In a $\delta$-hyperbolic space geodesics are coarsely determined by their end points on the boundary. For CAT(0) spaces which contain Euclidean flats this is easily seen to be false, if the geodesic happens to be $\delta$-slim however it is true.
 
\begin{lemma} \label{lemma:contractingflatsarebounded} Let $a$ be a $\delta$-slim bi-infinite geodesic. If $b$ is a bi-infinite geodesic which stays a bounded distance from $a$ then $b$ will be in the $2\delta$-neighborhood of $a$. 

\end{lemma}

\begin{proof} Left as an exercise to the reader.

\end{proof}

As a consequence of the bounded projection property for contracting geodesics Charney--Sultan proved in \cite{CS13} Proposition~3.7 that contracting geodesics have a strong visibility condition.

\begin{lemma}[Visibility] \label{lemma:visibility} If $X$ is a CAT(0) space and $a$ is a contracting geodesic then if $b$ is any geodesic in $X$ there is a bi-infinite geodesic from $b(\infty)$ to $a(\infty)$.
\end{lemma}

\begin{lemma} \label{lemma:infiniteslimness} If $X$ is a proper CAT(0) space and $a$ is a $\delta$-slim geodesic in $X$ then for any $x \in X$ the distance $d(\pi_{a}(x), [x,a(\infty))) \leq \delta$. This just extends the concept of $\delta$-slimness.

\end{lemma}

\begin{proof} This is just an application of the $\delta$-slim condition to the sequence of geodesics $[x,a(N)]$ and the Arzel\'{a}--Ascoli theorem

\end{proof}

The following lemma is simply Corollary~3.4 from \cite{BF09}. 

\begin{lemma}[Thin rectangles] \label{thinrectangles}Let $w,x,y,z$ be points such that the geodesic $[x,y]$ is $D$-contracting and $\pi_{[x,y]}(w) = x$ and $ \pi_{[x,y]}(z) = y$. Then there exists an $M > 0$ such that either $d(x,y) < M$ or $d([x,y],[w,z]) < M$, where $M$ depends only on $D$.

\end{lemma}

We are going to need a slightly beefier version of Lemma~\ref{thinrectangles} in the following proofs. We are going to have to require tighter control of the entire geodesic and we will let one of our end points be a point in the boundary.

A remark on the notation in the following lemma. The projection of a point
in the contracting boundary, $\alpha$, onto a D-contracting geodesic, $\gamma$, is a well defined notion. For more details and a definition see Remark~\ref{barycenter} and the discussion before it.

\begin{lemma} \label{betterthinrectangles} Let $\gamma$ be a $D$-contracting geodesic, $\alpha \in \partial_cX$ and $x \in X$. If $c$ is the geodesic from $x$ to $\pi_\gamma(x)$, $b$ is the geodesic from $\pi_\gamma(x)$ to $\pi_\gamma(\alpha)$ and $a$ is the geodesic from $\pi_\gamma(\alpha)$ to  $\alpha$, then there exists an $M\geq 0 $ such that either $d(\pi_\gamma(x),\pi_\gamma(\alpha)) < M$ or the geodesic $[x,\alpha)$ is in the $M$ neighborhood of $c\cup b \cup a$ and vise versa. 

\end{lemma}

\begin{proof}

\begin{figure}[h]
\begin{tikzpicture}

\draw[<->] (0,0) -- (9,0) node[pos=0,anchor=east]{$\gamma$} node[pos=0.2,fill,scale=0.3,circle](Px){} node[pos=0.2,anchor = north]{$\pi_\gamma(x)$} node[pos=0.8,fill,scale=0.3,circle](PA){} node[pos=0.8,anchor = north]{$\pi_\gamma(\alpha)$};

\node at ($(Px)+(-0.5,2.5)$)(x) [fill,circle,scale=0.3]{}; 	%node x
\node at (x) [anchor = south east]  {$x$};  				%place label

\node at ($(PA)+(-1,5)$)(A) {}; 							%node A
\node at (A) [anchor = south east]  {$\alpha$}; 			%place label

\draw (x) -- (Px);
\draw[->] (PA) -- (A) node [pos=0.75,scale = 0.3,fill,circle,black](w){} node [pos=0.75,anchor = west]{$w$};

\node at ($(Px)+(0.7,0.3)$)(j1) {};						%node j1

\draw (x) ..controls ($(j1)-(0.5,0)$).. (j1);

\node at ($(PA)-(0.7,-0.3)$)(j2){};						%node j2
\draw (w) ..controls ($(j2)+(0.5,0)$).. (j2);

\draw (j1.west) -- (j2.east) node[pos = 0.5,scale=0.3,circle,fill](z1){} node [pos=0.5,anchor = south]{$z_1$};

\node at ($0.5*(Px)+0.5*(PA)$) [scale = 0.3,circle,fill] (z2){}; %node z1
\node at (z2) [anchor = north]{$z_2$};							 %label z1

\draw[help lines] (z1) -- (z2);

\end{tikzpicture}
\caption{Bounding rectangles}
\label{fig:rectangles}
\end{figure}

First fix a $w$ on the geodesic $a$. We will prove the lemma replacing $\alpha$ with $w$ and that will suffice as you can take a sequence of $w$ tending towards $\alpha$ and apply Arzel\`a--Ascoli and obtain the lemma. 

Applying Lemma~\ref{thinrectangles} to the points $x,\pi_\gamma(x),\pi_\gamma(\alpha),w$ we get an $M'$ such that either there are points $z_1$ and $z_2$  on the geodesic $[x,w]$ such that $d(z_1,z_2) < M'$  or $d(\pi_\gamma(x),\pi_\gamma(\alpha)) < M'$ (see Figure~\ref{fig:rectangles}). We may assume the former. 

Since $\gamma$ is $D$ contracting we may assume that $[z_2,\pi_\gamma(\alpha)]$ is also $D$ contracting, and thus the triangle $z_2, w, \pi_\gamma(\alpha)$ is $\delta_D$ slim. Since $[z_1,w]$ is in the $M'$ neighborhood of $[z_2,w]$ and $[z_2,w]$ is in the $\delta$ neighborhood of $[z_2,\pi_\gamma(\alpha)]\cup [\pi_\gamma(\alpha),w]$, (where $\delta$ might be a linear function of $\delta_D$), we have that $[z_1,w]$ is in the $M'+\delta$ neighborhood of it as well. Running the argument in the other direction gives you that $[z_2,\pi_\gamma(\alpha)]\cup [\pi_\gamma(\alpha),w]$ is within the $M'+\delta$ neighborhood of $[z_1,w]$.

By repeating this argument with $x,\pi_\gamma(x),z_2$ and $z_1$ and setting $M = M'+\delta$ we get the result.

\end{proof}

This next lemma gives us information about the global geometry when the equivalence class of a contracting ray is fixed by a cocompact group action. This lemma will allow us to rule out the existence of global fixed points in the contracting boundary later on.

\begin{figure}[h]
\begin{tikzpicture}[thick]

	\node at (0,0)[anchor=north east]{$x$};
	\node[fill,circle,scale=0.3](x){};
	\node(ainfty) at (11,0){};
	\draw[->] (x) -- (ainfty) node[anchor=south west]{$\alpha$};
	\draw[->] (x) .. controls(3.5,2) and (3.5,3).. (4,7) node[pos=0.8,circle,fill,scale=0.3](bi){} node[pos=0.8,anchor=east]{$b(i)$} node[pos=1,anchor=south]{$b(\infty)$};
	\node at ($(bi)+(1,0)$)[circle,fill,scale=0.3](gix){};
	\node at (gix)[anchor=west]{$g_ix$};
	\node[circle,fill,scale=0.3](pibi) at ($(x)!(bi)!(ainfty)$){};
	\node[anchor=south east] at (pibi){$\pi_{a}(b(i))$};
	\node[circle,fill,scale=0.3](pigix) at ($(x)!(gix)!(ainfty)$){};
	\node[label={[label distance=4pt]3:$\pi_{a}(g_ix)$}] at (pigix){};
	\draw[help lines,dashed,thin] (bi) -- (pibi);
	\draw[help lines,dashed,thin] (gix) -- (pigix);
	\draw[->] (gix) ..controls ($(pigix)+(0.5,0.5)$) and ($(pigix)+(1.5,0.5)$) .. ($(ainfty)+(-0.16,0.5)$) node[pos=0.4,scale=0.3,circle,fill](w){};
	\draw[help lines] (pigix) -- (w) node[pos=0.7,anchor=south east,color=black]{$\delta_A$};
	\draw[decorate,decoration={brace,mirror,raise=0.1cm,amplitude=4pt}] (x) -- (pibi) node[midway,label={[label distance=5pt]270:$P$}]{};
	\draw[decorate,decoration={brace,mirror,raise=0.1cm,amplitude=4pt}] (pibi) -- (pigix) node[midway,label={[label distance=5pt]270:$C$}]{};
	\draw[help lines] (bi) -- (gix) node[pos=0.5,anchor=south,color=black]{$C$};

\end{tikzpicture}
\caption{A globally fixed contracting geodesic}
\label{fig:fixedimpliescontracting}
\end{figure}

\begin{lemma}  \label{lemma:fixedimpliescontracting} Let $G$ be some group acting cocompactly by isometries on a CAT(0) space $X$. If there is some $\alpha \in \partial X$ such that $G$ fixes $\alpha$ and some representative of $\alpha$ is contracting, then every ray in $X$ is contracting. 

\end{lemma}

\begin{proof} Let $b$ be some ray in $X$. Pick a representative $a$ of $\alpha$ such that $a(0) = b(0) = x$. Note that because one of the representatives of $\alpha$ is contracting, all of them are, though the contracting constant will depend on $x$, so let $A$ be such that $a$ and all subsegments of $a$ are $A$-contracting. By Lemma~\ref{lemma:boundedprojection} the projection of $b$ onto $a$ is bounded, i.e. there is a $P$ such that $d(x,\pi_{a}(b(i))) \leq P$ for all $i$. By cocompactness we also have a $C\geq 0$ and a collection $\{g_i\} \subseteq G$ with $d(g_ix,b(i)) \leq C$. This implies that 

\begin{align*} 
d(x,\pi_{a}(g_ix)) 	& \leq d(x,\pi_{a}(b(i))) + d(\pi_{a}(b(i)),\pi_{a}(g_ix)) \\
					& \leq P + d(b(i),g_ix) \\
					& \leq P + C \end{align*}

Where the second inequality is by the definition of $P$ and the fact that the projection function is non-increasing. 

Because the $g_i$ leave $\alpha$ fixed we have that $g_ia$ is the geodesic from $g_ix$ to $\alpha$. Since $a$ is contracting, by Lemma~\ref{lemma:infiniteslimness} there is a $\delta_A$ so that for all $i$ $d(\pi_{a}(g_ix),g_ia) \leq \delta_A$. We can then derive the following inequality: 

$$d(x,g_ia) \leq P + C + \delta_A$$

\vspace{12pt}

Thus there is some $N_i$ such that $d(x,g_ia(N_i)) \leq P + C + \delta_A$. 

\vspace{12pt}

Because all subsegments of $g_ia$ are also $A$-contracting we have that $[b(0),b(i)]$ is close to an $A$-contracting geodesic and  Lemma~\ref{lemma:closetocontracting} then implies that $[b(0),b(i)]$ is $\Phi_{\ref{lemma:closetocontracting}}(A,C,P+C+\delta_A)$-contracting for all $i$. The geodesic $b$ is then contracting since every initial segment is contracting with the same constant.

\end{proof}

\begin{definition} Let $X$ be a complete CAT(0) space. The angle $\angle(\alpha,\beta)$ between $\alpha,\beta \in \partial X$ is defined as 
$$\sup_{x \in X} \angle_x(\alpha,\beta)$$
Where $\angle_x(\alpha,\beta)$ is the Alexandrov angle between the two (unique) geodesics which start at $x$ and are in the equivalence class of $\alpha$ and $\beta$. The function $\angle(\cdot,\cdot)$ defines a metric on $\partial X$ making it a complete metric space. The associated length metric is called the \emph{Tits metric} and is denoted $d_T(\alpha,\beta)$.

\end{definition}

For further information on the Tits metric see \cite[Chapter II.9]{BH}.

\bigskip
The following is a result of Ballmann and Buyalo \cite[Proposition~1.10]{BB08} and it supplies us with a rank one isometry for all complete cocompact CAT(0) spaces which have a contracting ray.

\begin{proposition} Suppose $X$ is a cocompact CAT(0) space and $\partial X$ is non-empty then the following are equivalent.

\begin{conditions} 

\item X contains a periodic rank-one geodesic.

\item For each $\xi \in \partial X$ there is an $\eta \in \partial X$ with $d_T(\eta,\xi) > \pi$.

\end{conditions}

\end{proposition}

\bigskip

\begin{corollary} \label{therearerankoneaxes} Let $X$ be a complete and proper CAT(0) space and let $G$ act on $X$ geometrically. If $X$ has a contracting ray then there is a rank-1 isometry. \end{corollary}

\begin{proof} By the strong visibility condition in Lemma~\ref{lemma:visibility}, if $X$ has a contracting ray $a$ then it is visible from all points $\xi \in \partial X$. The geodesic between $a(\infty)$ and any $\xi_0 \in \partial X$ guaranteed by visibility tells us that the Alexandrov angle $\angle(a(\infty),\xi_0) = \pi$. To show that the Tits distance is larger than $\pi$ from any point $\xi \in \partial X$, pick a geodesic from $a(\infty)$ to $\xi$ \emph{inside} the Tits boundary $\partial X$ and call it $c$ (note: if no such geodesic exists then $d_T(a(\infty),\xi) = \infty$). Now let $\xi_0$ be a point on $c$ separate from $a(\infty)$ and $\xi$. Then we will have that $d_T(a(\infty),\xi) = length(c) \geq \angle(a(\infty),\xi_0) + \angle(\xi_0,\xi) \geq \pi + \epsilon$, note that $\epsilon > 0$ since $\xi_0 \neq \xi$. So $d_T(a(\infty),\xi) > \pi$ and we have that $X$ has a rank-one periodic geodesic.
\end{proof}

\bigskip
We need the following technical fact about geodesics in metric spaces at several points in this paper, we include a proof for the sake of completeness. 

\begin{lemma} \label{boundedtrianglemagic} Let $\gamma$ be a geodesic in a metric space $X$ and let $x$ be a point in $X$ such that $d(x,\gamma(0)) = t_0$, then if the distance $d(x,\gamma) \leq D$ then $d(x,\gamma(t_0))\leq 2D$. 

\end{lemma} 

\begin{proof} Since $d(x,\gamma) \leq D$ let's let $\ell$ be a point such that $d(x,\gamma(\ell)) \leq D$. There are two cases, $\ell \geq t_0$ or $\ell < t_0$.

\bigskip
In the first case if we consider the geodesic triangle defined by $\gamma(0), \gamma(\ell)$, and $x$, but let's rewrite $\ell = t_0+a$. The triangle inequality says that $t_0+a \leq D + t_0$ i.e. $a \leq D$. Then considering the triangle defined by the three points $\gamma(t_0+a),\gamma(t_0)$, and $x$ we get a new triangle inequality $d(x,\gamma(t_0)) \leq a + D \leq 2D$.

\bigskip
In the second case we will again consider the geodesic triangle given by $\gamma(0), \gamma(\ell)$, and $x$ but this time and we will write $\ell = t_0-a$. The triangle inequality fashions us with $t_0  \leq D + (t_0-a)$ or $a \leq D$. Considering the triangle defined by $\gamma(t_0-a), \gamma(t_0)$, and $x$ we get the triangle inequality $d(x,\gamma(t_0)) \leq D + a \leq 2D$.

\end{proof}

\section{The topology of the contracting boundary}
\label{sec:topology}

The topology of the contracting boundary is very different from any of the typical topologies put on the visual boundary. Later, in Section~\ref{sec:metrizability}, we will show that the contracting boundary is not always a metric space. In fact, we show it to not even be first-countable. In anticipation of that we will prove some elementary topological facts about the contracting boundary (and direct limit spaces in general) to facilitate some of the later proofs.

First, let's define the contracting boundary and then we will talk about some of its basic topological properties.

\begin{definition} Let $X$ be a CAT(0) space. Let $\partial_c^DX_x$ be the set of infinite geodesic rays that start at $x$ and are $D$-contracting, we shall call this the \emph{D-component} of the contracting boundary. This is a subspace of the visual boundary of $X$, $\partial X_x$, and has the associated topology on it. If $D_0 \leq D_1$ then there is the natural continuous inclusion $\partial_c^{D_0}X_x \hookrightarrow \partial_c^{D_1}X_x$,  so taking all non-negative $D$ we get a directed system. 

The \emph{contracting boundary}, denoted $\partial_cX_x$, is the union of all of the D-components with the direct limit topology. 

\end{definition}

The homeomorphism type (but not the contracting constants) of the contracting boundary is independent of the base point $x$, and so typically this will be suppressed when there is no danger of confusion \cite{CS13}. 

One of the basic properties of a direct limit space is that a set in the space is open (respectively closed) if and only if its intersection with each component is open (closed). In fact, this is often taken as the definition. 

Because the topology of the contracting boundary is so dependent on the topology of the components it will be useful to know how the subspace topology on the components sits inside of the visual topology. The following is Lemma~3.3 in \cite{CS13}.

\begin{lemma} \label{lem:componentsareclosed} For all $D \geq 0$ the $D$-components of the contracting boundary are closed subsets of the visual boundary. 

\end{lemma}

Understanding compact sets in the contracting boundary will be important later in our investigation. It turns out that compact sets in the contracting boundary are closely related to the compact sets of the visual boundary, but are limited by their contracting constants. 

\begin{lemma} \label{compactsets} A set $K$ is compact in $\partial_cX$ if and only if $K = C \cap \partial_c^DX$ for some compact set $C \subset \partial X$ and some $D$.

\end{lemma}

\begin{proof} $\Leftarrow$ If $K = C\cap \partial_c^DX$ then because $C$ is compact in $\partial X$ and $\partial_c^DX$ is a closed set in $\partial X$ by Lemma~\ref{lem:componentsareclosed}, then $K$ is a closed subset in $C$ and therefore compact in $\partial X$. Now the topology on $\partial_c X$ is defined in such a way so that each of the components $\partial_c^DX$ are topologically embedded into $\partial_cX$, i.e. compact subsets of $\partial_c^DX$ will also be compact in $\partial_cX$. So $K$ is a compact set in $\partial_cX$.

\bigskip
$\Rightarrow$ Assume that $K$ is a set in $\partial_cX$ but that $K$ is not contained in $\partial_c^DX$ for any $D$. These assumptions guarantee that there is some sequence of geodesics $\{a_i\}$ in $K$ where $a_i$ is $D_i$-contracting and $D_i \rightarrow \infty$. By possibly passing to a subsequence we may assume that $D_i > D_{i-1}$ and that each $a_i$ is not $D_{i-1}$-contracting. Let $\displaystyle A_n = \{a_i\}_{i \geq n+1}$. Note that for all $n$ and all $D$, $A_n\cap \partial_c^DX$ is a finite set and therefore closed in each component, so $A_n$ is closed in $\partial_cX$. 

The collection $\mathcal{O} = \{\partial_cX \backslash A_n\}$ is an open cover of $K$, but each open set only contains finitely many of the $a_i$. Take any finite subcollection of $\mathcal{O}$, it will only cover finitely many of the $a_i$ and so it is not a cover, therefore $K$ is not compact. We can then conclude that if $K$ is compact, it is contained in one of the components $\partial_c^DX$ for some $D$. Because the topology on $\partial_cX$ is finer than that of $\partial X$ any set which is compact in the contracting boundary is compact in the visual boundary. In other words every compact set $K$ in the contracting boundary is of the form $K = K \cap \partial_c^DX$ for some $D$. 

\end{proof}

\bigskip

We will also want to know when sequences in the contracting boundary converge. It turns out that a sequence converges in the contracting boundary if and only if it converges in the visual boundary and its contracting constants are uniformly bounded above.

\begin{lemma} \label{convergenceincontractingboundary} Let $X$ be a proper CAT(0) metric space. A sequence $a_i$ in $\partial_cX_x$ converges to a point $b \in \partial_cX_x$ if and only if the following two conditions hold:

\begin{conditions}

\item There is a uniform $K$ such that for all $i$ $a_i$ is $K$-contracting.

\item In the visual boundary $a_i \to b$.

\end{conditions}

\end{lemma}

\begin{proof} $\Leftarrow$ Since $a_i$ are all $K$-contracting then the set $\{a_i,b\} \subseteq \partial_c^MX_x$ where $M$ is the max of the contracting constant of $b$ and $K$. The topology on this component is just the subspace topology and thus since the $a_i \to b$ in the visual boundary the convergence happens in this component as well. Because each of these components are topologically embedded the convergence takes place in $\partial_cX_x$ as well.

$\Rightarrow$ Note that the topology on $\partial_cX_x$ is finer than that of the subspace topology. In particular, if condition (2) fails then $a_i$ will not converge to $b$ in the contracting boundary. 

Assume (1) fails, this means that for each $k \in \mathbb{N}$ there is an $i_k$ such that $a_{i_k}$ is not $k$ contracting. Consider the set $\{a_{i_k}\}$, this subsequence is in fact closed in $\partial_cX_x$ because only finitely many of them are in each $\partial_c^DX_x$ and are thus closed in the subspace topology. Thus $a_i \not\rightarrow b$ in $\partial_cX_x$.

\end{proof}

For a point in the contracting boundary, $\alpha \in \partial_cX$, and a $C$-contracting geodesic ray $\gamma$, whose forward endpoint is different from $\alpha$, we can define the projection of $\alpha$ onto $\gamma$, $\pi_\gamma(\alpha)$. If we take a representative of $\alpha$, say a geodesic $a$, the projection of $a$ onto $\gamma$ is a set of finite diameter by Lemma~\ref{lemma:boundedprojection}, there is then some unbounded sequence of $t_i$ such that $\pi_\gamma(a(t_i))$ will converge to some point $\gamma(T)$. Note that the point $\gamma(T)$ depends, not only on the chosen representative of $\alpha$, but also on the sequence of $t_i$. For a given representative $a$ we have that $a$ is eventually contained in the compliment of any bounded neighborhood of $\gamma$. Applying Lemma~\ref{thinrectangles}, given a large enough $t_0$ the geodesic $[a(t_0),a(\infty))$ is not contained in the $M_C$ neighborhood of $\gamma$ and so for all $t,t' > t_0$, $d(\pi_\gamma(a(t)),\pi_\gamma(a(t')))< M_C$. Thus for any two sequences $a(t_i)$ and $a(t_i')$, since the projections of these two sequences are eventually within $M_C$ of each other the limits are as well. If another representative of $\alpha$ is chosen, say $a'$, it is within a bounded distance of $a$, let us call that distance $D$. By picking $t_0$ large enough the geodesics $[a(t_0),a(\infty))$ and $[a'(t_0),a'(\infty))$ are outside of the $ D+M_C$ neighborhood of $\gamma$. Let $t > t_0$, there is a $t_1$ such that $d(a'(t),a(t_1)) \leq D$, so applying Lemma~\ref{thinrectangles} the projection of $a'(t)$ and $a(t_1)$ are within $M_C$ of each other. Thus the projections of $[a'(t_0),a'(\infty))$ and $ [a(t_0),a(\infty))$ onto $\gamma$ are within $2M_C$ of each other. 

\begin{remark} \label{barycenter} A closed and bounded set in a proper CAT(0) metric space has a
unique center. That is, a point which is the center of the smallest circle which
contains the entire set exists and is unique. \cite[II.2.7]{BH}. In this paper that point will
be referred to as the \emph{barycenter}.
\end{remark}

\begin{definition} Given $\alpha \in \partial_cX$, and a $C$-contracting ray $\gamma$, by the previous discussion the set  $E = \left\{ \; \displaystyle \lim _{i\to \infty } \pi_\gamma a(t_i) \; | \; a \in \alpha \mbox{ and the limit exists}  \right \} $  will be non-empty and have diameter at most $2M_C$. Let $\pi_\gamma(\alpha)$ be the barycenter of $E$ (or the closure of $E$ if necessary).
\end{definition}

\begin{remark} For any $T$ the three points, $\gamma(T), \alpha, \pi_\gamma(\alpha)$ form an infinite $\delta$-slim triangle where $\delta$ depends only on $C$.

\end{remark}

We will topologize the set $\overline{X}_c = X \cup \partial_cX$. Recall that $X$ can be re-defined as the set of so called ``generalized'' rays in $X$. A \emph{generalized} ray is a map $a: [0,\infty) \to X$ such that an initial component $a|_{[0,t)}$ is an isometric embedding and the map $a|_{[t,\infty)}$ is constant (by setting $t = \infty$ we get back our infinite rays). Fixing a base point, $x_0$, such that all $a(0) = x_0$ each point $x \in X$ is represented by the unique generalized ray, $a_x$, whose initial component is the geodesic from $x_0$ to $x$, and is the constant function $a_x(k) = x$ for $k \geq d(x_0,x)$ otherwise. Let us denote such a representation of $X$ by $X_{x_0}$. This defines a topological space $\overline{X} = X \cup \partial X$ which is endowed with the cone topology.

\begin{definition}[Topology of $\overline{X}_c$] Let $x_0$ be a base point in $X$. Define the following set of generalized rays:

$$ \overline{X}^{D}_{x_0}= \{ c \in \overline{X} \; | \; c(0) = x_0 \; \mbox{ and } c \mbox{ is at most } D \mbox{-contracting}\}  $$

Endowing these sets with the subspace topology from $\overline{X}$ the inclusions will form a directed system. This gives us our topology on $\overline{X}_c$ as the direct limit.
 
 $$  \overline{X}_{c,x_0} := \lim_{\longrightarrow} \overline{X}^{D}_{x_0}$$ 

\end{definition}

\begin{remark} For all $D<D'$ we have the following commutative diagrams, where all inclusions are topological embeddings:

\begin{centering}

\begin{tikzcd}
	\partial_c^DX_{x_0} \arrow[hookrightarrow]{r} \arrow[hookrightarrow]{d}
		&  \partial_c^{D'}X_{x_0} \arrow[hookrightarrow]{d}\\
		\overline{X}^{D}_{x_0} \arrow[hookrightarrow]{r}
		& 	\overline{X}_{x_0}^{D'}
\end{tikzcd}

\end{centering}

By the universal property of the direct limit topology this implies that there is a continuous injection $\partial_cX_{x_0} \hookrightarrow \overline{X}_{c,x_0}$. Because all of the maps in the diagram are topological embeddings it is immediate that this map is also a topological embedding.

$X$ is also topologically embedded in $\overline{X}_{c,x_0}$. This is because, for each $x \in X$, every open ball is eventually contained in $\overline{X}_c^D$ for large enough $D$. This is just a consequence of Lemma~\ref{lemma:closetocontracting}. Further more $X$ is an open set in $\overline{X}_c$ and so $\partial_cX$ is a closed set.

\end{remark}

\begin{lemma}\label{sequenceconvergence} Consider a sequence of points $x_i \in X$. Fixing a base point $x$ the sequence $x_i$ converges to some $\alpha \in \partial_cX$ in the topology on $\overline{X}_c$ if and only if the geodesics $[x,x_i]$ converge to $\alpha$ in the cone topology on $\overline{X}$ and the contracting constants of the $[x,x_i]$ are uniformly bounded.

\end{lemma}

\begin{proof} The proof of this is the same as the proof of Lemma~\ref{convergenceincontractingboundary}.

\end{proof}

\section{The topological dynamics of the action on the boundary}
\label{sec:dynamics}

The topological dynamics of a group action can be a powerful tool in understanding the global topology. In order to better gain an understanding of the contracting boundary of cocompact CAT(0) spaces we will attempt to exploit some well known results from $\delta$-hyperbolic spaces. 

Most of the following results are well known dynamical results for the visual boundary of a CAT(0) group which contains a rank-1 isometry and a result of Ballmann and Buyalo's \cite{BB08} guarantees this is the case for the cocompact groups we are considering here. However, the definition of the topology as a direct limit of spaces gives the contracting boundary a much finer topology than the subspace topology would. Because of this different topology we are considering, it is necessary to reprove (and in some cases reword) these dynamics results as none of them will follow as immediate corollaries from the known theorems.

For non-elementary hyperbolic groups the orbit of every point in the boundary is dense. This establishes a strong dichotomy: either the group is virtually $\mathbb{Z}$ or its boundary has no isolated points.

Our first theorem is establishing this result in the case of the contracting boundary, i.e. the contracting boundary either has no isolated points and has a countable dense subset, or the group is virtually cyclic.

\begin{theorem} \label{thm:orbitsaredense} Let $X$ be a proper CAT(0) space such that $G$ acts geometrically on $X$. If $\partial_cX \neq \emptyset$ and $G$ is not virtually cyclic then the orbit of each point in $\partial_cX$ is dense. 

\end{theorem}

This is very similar to a result of Hamenst\"adt on the limit set when the group contains a rank one isometry \cite{H09}.

In the spirit of treating the contracting boundary as a replacement of the Gromov boundary for CAT(0) spaces it is natural to ask is whether axial isometries act with North-South Dynamics and if the group $G$ acts as a convergence group action on $\partial_cX$. Recall that an axial isometry is an isometry that fixes a geodesic, called the axis of the isometry, these are also called loxodromic or hyperbolic isometries in the literature. Because the contracting boundary is not compact the classical formulations of these dynamical properties will have to be reinterpreted somewhat. 

\begin{theorem} \label{thm:conv} Let $X$ be a proper CAT(0) space on which $G$ acts geometrically. Let $g_i$ be a sequence of isometries in $G$ such that $g_ix \to \gamma^+$ where $\gamma^+ \in \partial_cX$, then there is a subsequence of $g_i$'s where $g_i^{-1}x \to \gamma^-$ for some $\gamma^- \in \partial_cX$ and for every open neighborhood $U$ of $\gamma^+$ and every compact set $K \subseteq \partial_cX - \gamma^-$ we have uniform convergence of $g_i(K) \to \gamma^+$.

\end{theorem}

This theorem is closer to Papasoglu and Swenson's $\pi$-convergence from \cite{PS} than it is to a true convergence action. A corollary of this theorem is that rank-1 isometries act with a version of North-South dynamics on the contracting boundary.

\begin{corollary} Let $X$ be a proper CAT(0) space and let $G$ be a group acting geometrically on it. If $g$ is a rank-1 isometry in $G$, $U$ is an open neighborhood of $g^{\infty}$ and $K$ is a compact set in $\partial_cX - g^{-\infty}$ then for sufficiently large $n$, $g^{n}(K) \subseteq U$.

\end{corollary}

\subsection{Failure of classical North-South dynamics}

By classical North-South dynamics we mean the following theorem. 

\begin{theorem} If $G$ is a $\delta$-hyperbolic group acting on its Cayley graph $X$ and if $g$ is an infinite order element then for all open sets $U$ and $V$ with $g^\infty \in U$ and $g^{-\infty} \in V$ then $g^n(V^c) \subseteq U$ for large enough $n$. 

\end{theorem}

It is a well established fact that for CAT(0) groups the classical version of North-South dynamics of axial isometries on the visual boundary fails. In particular, if the isometry is not rank-one whole flats may be fixed by the isometry.

\bigskip

Unfortunately, even if $g$ is a rank-1 element of $G$, this classical version of North-South dynamics on $\partial_cX$ still fails. If $a$ is an axis for $g$ there are open sets $U$ and $V$ of $a(\infty)$ and $a(-\infty)$ such that $g^N(\partial_cX \backslash V) \not\subseteq  U$ for any $N$.

Note: This is in direct contrast with the subspace topology on the set of contracting geodesics, ($\partial^{sub}_c X$). In \cite{H09} and \cite{B95} it was proven that rank-1 isometries act on the entire visual boundary with North-South dynamics and thus on any subspace containing the end points.

\bigskip

For an example of the failure of the classical North-South dynamics of rank-1 isometries on the contracting boundary consider the RAAG, $A_\Gamma = \langle a,b,c  \; | \; [b,c] \rangle$. This is the fundamental group of the Salvetti complex, X (see figure \ref{fig:salvetti}), and its universal cover, $\tilde{X}$, is a CAT(0) cube complex on which $A_\Gamma$ acts geometrically \cite{CD95}. Let $\gamma$ be an axis for the loxodromic element $a$. Let $b_i$ be the geodesics following the words $a^{-i}b^iaaaa\cdots$. Note that the contracting geodesics $b_i$ do \emph{not} converge to $\gamma(-\infty)$ in the contracting boundary. This is because the intersection of the set $\{b_i\}$ with each of the contracting components $\partial_c^DX$ is a finite set and therefore closed in the subspace topology, and thus $\{b_i\}$ is closed in $\partial_cX$. 

\bigskip
\bigskip

\begin{figure}[h]
\begin{tikzpicture}

	\draw (0,0) ellipse (1.5cm and 0.7cm);
	\draw (1.5,0) node(x){} node[anchor=west]{$x$};
	\draw[color=red] ($(x)+(0.54,0)$) ellipse (0.54cm and 0.3cm);
	\node at ($(x)+(0.54,0)+(0:0.54cm and 0.3 cm)$)[anchor=west,color=red]{$a$};
	\draw[color=cyan] (x) arc (0:180:0.54cm and 0.3cm) node[pos=0.5,anchor=north]{$b$};
	\draw[color=cyan,dashed] (x) arc (0:-180:0.54cm and 0.3cm);
	\draw[color=orange] (x) arc (0:-180:1.5cm and 0.5cm) node[pos=0.5,anchor=south]{$c$};
	\draw[color=orange,dashed] (x) arc (0:180:1.5cm and 0.5cm);
	\draw (0.42,0) arc (65:115:1cm);
	\draw (0.5,0.04) arc (-60:-120:1cm);
	\node[fill,circle,scale=0.3] at (x){};

\end{tikzpicture}
\caption{The Salvetti complex $X$ of $A_\Gamma$.}
\label{fig:salvetti}
\end{figure}

\bigskip

The set $V = \left(U_{\gamma^-}(r,\epsilon)\cap\partial_cX\right) \setminus \{b_i\}$ is then an open set around $\gamma(-\infty)$ but for all $N$ we have $a^Nb_{N} \not\in U_{\gamma^+}(r',\epsilon')$ for all $\epsilon' < r'$.

\subsection{Proof of Theorem \ref{thm:orbitsaredense}}
\label{sec:orbitsaredense}

The first step in proving this theorem will be to prove an initially weaker result. We will prove that for a cocompact CAT(0) space, the orbit of a point in the contracting boundary is either a singleton or is dense. The proof relies on the observation that the orbit of a bi-infinite geodesic is easier to understand and contains more geometric information than the orbit of an infinite ray. We will take some contracting ray and one of its orbit points and connect the two with a bi-infinite geodesic. It is then reasonably easy show that the orbit of this bi-infinite geodesic is dense in the contracting boundary.

\begin{proposition} \label{prop:orbitsaredense} If the action of $G$ on $X$ is cocompact and $\alpha^+ \in \partial_cX$ then $\alpha^+$ is globally fixed by $G$ or its orbit is dense in $\partial_{c}X$. \end{proposition}

\begin{proof} 

First note that if there are only two points in $\partial_cX$ then the proposition is obvious. Either the orbit is a singleton or it is the entire boundary. So from now on we may assume that $|\partial_cX| > 2$ and that $\alpha^+$ isn't globally fixed. 

To show that the orbit is dense it suffices to show that for all $\beta \in \partial_cX$ there exists a sequence of $g_i \in G$ such that $g_i\alpha^+ \to \beta$.  

If $\beta \in G\alpha^+$ then we are done since there is an $h$ such that $\beta = h\alpha^+$ so the constant sequence $g_i = h$ will work. 

If $\beta$ is not in the orbit of $\alpha^+$ pick a point distinct from $\alpha^+$ in $G\alpha^+$ and call it $\alpha^-$, i.e. $h\alpha^+ = \alpha^-$ for some $h \neq e$. By the visibility of $\partial_cX$ there is a geodesic connecting $\alpha^-$ to $\alpha^+$. If we label this geodesic $a$ and pick a base point $x = a(0)$ on it there is also a representative of $\beta$, $b$, such that $b(0) = x$. 

\vspace{12pt}
\emph{Note:} Since $\alpha^+$ and $\alpha^-$ are different elements of the contracting boundary there are two different contracting constants for their representatives $a|_{[0,\infty)}$ and $a|_{(-\infty,0]}$, but by Lemma~\ref{lemma:concatenationsarecontracting} we have a uniform contracting constant for all of $a$ and we shall call it $A$. For the representative, $b$, of $\beta$ let $B$ be its contracting constant. Since Lemma~\ref{lemma:slimgeodesics} guarantees that $a$ and $b$ are slim, we will denote $\delta_A$ and $\delta_B$ as their slimness constants respectively. To make the following discussion simpler, we will assume that $A$ and $B$ are chosen so that all subsegments (finite or infinite) of either geodesic are also contracting with the same constant.

\vspace{12pt}
By the cocompactness of the action of $G$ on $X$ there is a uniform $C > 0$ such that for each $i \in \mathbb{N}$ there is a $g_i \in G$ such that $d(g_ix,b(i)) \leq C$. Since we've picked $g_i$ so that the orbit of $x$ travels up along $b$ we'd like to say that the geodesic $a$ follows suit, but first we need to pass to a subsequence. 

Let $g_ia$ be the bi-infinite geodesic connecting $g_i\alpha^-$ to $g_i\alpha^+$ with base point $g_ix$. Now note that there is a $t_i$ such that $g_ia(t_i)$ is the projection of $x$ on $g_ia$ (see Figure~\ref{fig:converge}).

\begin{figure}
	\includegraphics[clip=true,trim=1in 7.8in 1in 1in]{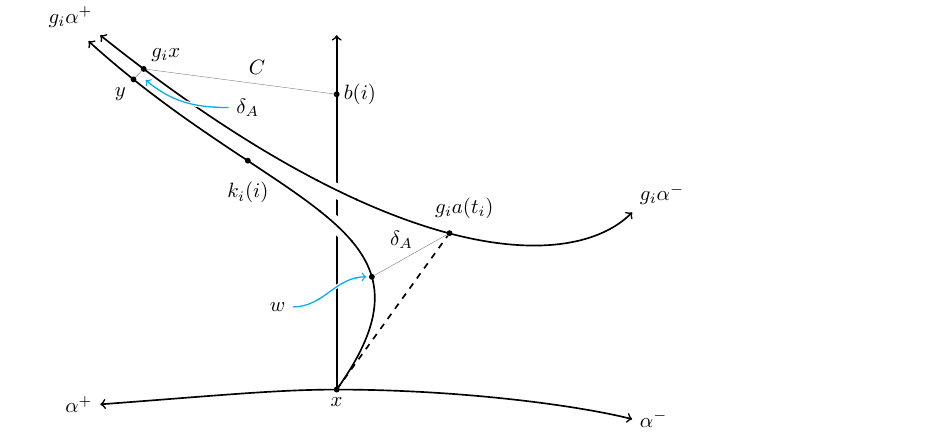}
\caption{Convergence of $\alpha$ translates} \label{fig:converge}
\end{figure}

Infinitely many of the $t_i$ will be either positive or negative, so by passing to a subsequence we may assume that all the $t_i$ have the same sign.

\vspace{12pt}
In the following argument we will consider the case when the $t_i \leq 0$. In this case we will prove that $g_i\alpha^+ \to \beta$. If instead, the $t_i > 0$ the following argument will go through, \emph{mutatis mutandis}, to show that $g_i\alpha^- \to \beta $. Because $\alpha^- = h\alpha^+$, this tells us $g_ih\alpha^+ \to \beta$. Thus, in either case, the orbit of $\alpha^+$ will accumulate on any $\beta \in \partial_cX$.

\vspace{12pt}
Consider the representatives of $g_i\alpha^+$ starting from the base point $x$ and denote them $k_i$. To show that the sequence $k_i$ converges in $\partial_cX$ to $b$, Lemma~\ref{convergenceincontractingboundary} says we only need the following two conditions: 

\vspace{12pt}
\begin{conditions}  
\item \label{condition1} There is a uniform $K$ such that for all $i$, $k_i$ is $K$-contracting.

\item \label{condition2} $k_i$ converges to $b$ in the visual boundary $\partial X$. 
\end{conditions}

\vspace{12pt}
It turns out that these two ingredients are a direct consequence of the following lemma:

\begin{lemma} \label{kiiisclosetobi} There is a constant $C$ such that for each $i$ the following holds: 

$$d(k_i(i),b(i)) \leq 2(\delta_A + C)$$

\vspace{12pt}

\end{lemma}

\begin{proof}[Proof of Lemma~\ref{kiiisclosetobi}] For the following discussion see Figure \ref{fig:converge}. We only need to show that the distance from the point $b(i)$ to the geodesic $k_i$ is $\delta_A + C$ then applying Lemma~\ref{boundedtrianglemagic} we get the result. 

Observe that $g_ia$ is $A$ contracting and thus there is a point $w$ on the geodesic $k_i$ which is within $\delta_A$ of $\pi_{g_ia}(x)$ by Lemma~\ref{lemma:infiniteslimness}.  Recall that $\pi_{g_ia}(x) = g_ia(t_i)$ and that $t_i \leq 0$. By the convexity of the distance function any point along the geodesic $g_ia|_{[t_i,\infty)}$ will also be within $\delta_A$ of $k_i$. In particular, since $g_ix = g_ia(0) \in g_ia|_{[t_i,\infty)}$ then $d(g_ix,k_i) \leq \delta_A$. 

Because of how the $g_i$ were defined we also have that $d(g_ix,b(i)) \leq C$. This lets us conclude that $d(b(i),k_i) \leq \delta_A + C$.

\end{proof}

Lemma~\ref{kiiisclosetobi} is the key to establish conditions \ref{condition1} and \ref{condition2}.

\vspace{12pt}
\begin{proof}[Proof of condition \ref{condition1}] Since $d(k_i(i),b(i)) \leq 2(\delta_A + C)$, which is independent of $i$, so by Lemma~\ref{lemma:closetocontracting} there exists a constant independent from $i$, $\Phi_{\ref{lemma:closetocontracting}}$ such that the geodesic $k_i|_{[0,i]}$ is $\Phi_{\ref{lemma:closetocontracting}}$-contracting. 

\vspace{12pt}
The cocompact constant gives us that $d(b(i),g_ix) \leq C$ so together with Lemma~\ref{kiiisclosetobi} we have $d(k_i(i),g_ix) \leq 2\delta_A + 3C$.

\vspace{12pt}
Because $g_ia|_{[t_i,\infty)}$ is $\delta_A$-slim, for large enough $T$ the point $k_i(T)$ is within $\delta_A$ of $g_ia|_{[t_i,\infty)}$. You can apply Lemma~\ref{lemma:closetocontracting} again to all subsegments $k_i|_{[i,T]}$ with large $T$, thus they are all $\Phi_{\ref{lemma:closetocontracting}}'$-contracting for some $\Phi_{\ref{lemma:closetocontracting}}'$ independent of $i$. This implies that the infinite ray $k_i|_{[i,\infty)}$ is contracting with the same contracting constant.

\vspace{12pt}
The concatenation of $k_i|_{[0,i]}$ and $k_i|_{[i,\infty)}$ gives us the entire geodesic $k_i$. Lemma~\ref{lemma:concatenationsarecontracting} then tells us that for all $i$, the $k_i$ are $\left(\Phi_{\ref{lemma:closetocontracting}} + \Phi_{\ref{lemma:closetocontracting}}'\right)$-contracting.

\end{proof}

\vspace{12pt}
\begin{proof}[Proof of condition \ref{condition2}] Recall that the sets 

$$U_{b}(\varepsilon,r) = \{c \; |\; c(0) = x \mbox{ and } d(c(r),b(r)) < \varepsilon\}$$

\vspace{12pt}
\noindent form a local neighborhood basis for the visual boundary. So for each $U_{b}(\varepsilon,r)$ we need an $N$ such that $k_i \in U_{b}(\varepsilon,r)$ for $i \geq N$. 

$$\displaystyle N(\varepsilon,r) := \max\left\{r \; , \; \; \frac{2r(\delta_A + C)}{\varepsilon}\right\}$$

\vspace{12pt}
\noindent is just such an $N$. When $i \geq N(\varepsilon,r)$ we get the following chain of inequalities:

$$ d(k_i(r),b(r)) \leq \frac{r}{i}d(k_i(i),b(i)) \leq \frac{r}{i}2(\delta_A + C) \leq \varepsilon$$

\vspace{12pt}
The first inequality is just a restatement of the convexity of the distance function (and is the reason $N(\varepsilon,r)$ is chosen as a max), the second is a result of Lemma~\ref{kiiisclosetobi} and the final inequality is just a restatement of the definition of $N(\varepsilon,r)$. Thus we have that the sequence $k_i$ converges to $b$ in the visual boundary.

\end{proof}

Establishing conditions \ref{condition1} and \ref{condition2} tells us that $k_i \to b$ in $\partial_cX_x$. Because $b$ was arbitrary this tells us that the orbit $G\alpha^+$ is dense in $\partial_cX_x$ and so the statement of the proposition is proven.

\end{proof}

The following corollary will come up later and so we will include it here. It states that the orbits of contracting things which aren't globally fixed are dense in the \emph{visual} boundary.

\begin{corollary} \label{por:orbitsarevisuallydense} If $G$ acts cocompactly on $X$ and $\alpha^+ \in \partial_cX$ isn't globally fixed by $G$ then its orbit is dense in $\partial X$. 

\end{corollary}

\begin{proof} This is an immediate consequence of the proof of condition $(2)$ in the above. At no point was the contracting constant of $b$ used and so replacing it with a non-contracting geodesic gives the same result. (Note that in this case condition \ref{condition1} fails).

\end{proof}

Proposition \ref{prop:orbitsaredense} is the major component of Theorem~\ref{thm:orbitsaredense}, but there remain a few loose ends. Here is an outline what remains of the proof. We need to first show that there are enough contracting geodesics in any cocompact CAT(0) space, namely that if the contracting boundary is not empty it contains at least 2 points. Second, we need to show that if there are exactly two points in the contracting boundary the group is virtually cyclic. This will establish our dicotomy, that our group is virtually cyclic or there are strictly more than two points in our contracting boundary. Finally, it will be easy to then show that if there are more than two points in the contracting boundary, none of them are globally fixed.

\begin{proposition} If $G$ acts geometrically on a proper CAT(0) space $X$ then $|\partial_cX| = 2$ if and only if $G$ is virtually $\mathbb{Z}$. 

\end{proposition}

\begin{proof} $\Rightarrow$ Let $a$ be a contracting geodesic connecting the two points in $\partial_cX$. Recall that this implies that $a$ is $\delta$-slim for some $\delta$. Because the action of $G$ on $X$ is cocompact there is some $C$ such that for all points $x \in X$ there is some $g_x$ such that $d(g_xa(0),x) \leq C$. Because the contracting boundary only contains two points then $g_xa$ is a bi-infinite geodesic which is asymptotic to the bi-infinite geodesic $a$. By Lemma~\ref{lemma:contractingflatsarebounded} we have that $d(g_xa,a) \leq 2\delta$ so the distance between $x$ and $a$ is bounded by $2\delta+C$. Thus $a$ is a quasi-surjective quasi-isometric embedding of $\mathbb{R}$, i.e. $X$ is quasi-isometric to the real line, and thus $G$ is QI to $\mathbb{Z}$. It is a standard exercise to show that a group which is QI to $\mathbb{Z}$ is virtually cyclic. For a sketch of the proof see \cite[pg 10 exercise 1.16]{GH}

$\Leftarrow$ If $G$ is virtually $\mathbb{Z}$ then it is QI to $\mathbb{R}$. The contracting boundary of a CAT(0) space is a QI invariant, so $\partial_cX  = \partial_c\mathbb{R} $ which is two discrete points. 

\end{proof}

\begin{lemma} If $X$ is a proper CAT(0) space with a geometric action and $\partial_cX \neq \varnothing$ then $|\partial_cX| \geq 2$.

\end{lemma}

\begin{proof} Since the contracting boundary is non-empty we have at least one contracting ray $a$, now look at the orbit of $a$, if it is not fixed we're done since the orbit of a contracting ray is contracting. If it is fixed then by Lemma~\ref{lemma:fixedimpliescontracting} every geodesic ray is contracting. So now the only way that we wouldn't have at least two points in the contracting boundary was if all infinite geodesics were asymptotic. However, if a CAT(0) group is not finite, it contains an infinite order element which has an axis in $X$, for a proof see \cite{S99}.

\end{proof}

\begin{proposition}[The Flat Plane Theorem] \label{flatplanetheorem}If a group $G$ is acting geometrically on a CAT(0) space, $X$, then $X$ is $\delta$-hyperbolic if and only if $X$ contains no Euclidean flats $\mathbb{E}^2$. 

\end{proposition}

This is a standard result from \cite[III.H.1.5]{BH}.

\begin{corollary} \label{nofixedpoints} Let $G$ act geometrically on a proper CAT(0) space $X$ with non-empty contracting boundary $\partial_cX$. If $G$ fixes a point in $\partial_cX$ then $G$ is virtually $\mathbb{Z}$.

\end{corollary}

\begin{proof} Let $\alpha \in \partial_cX$ be a fixed point. By Lemma~\ref{lemma:fixedimpliescontracting} we have that every geodesic in $X$ is contracting. In particular, we have that $X$ cannot contain a Euclidean flat and thus by The Flat Plane Theorem \ref{flatplanetheorem} $X$ is $\delta$-hyperbolic. \v{S}varc-Milnor then tells us that $G$ is a $\delta$-hyperbolic group. Note that in this case the contracting boundary is the Gromov boundary.

Recall that if a $\delta$-hyperbolic group is non-elementary i.e it is neither finite nor virtually cyclic, then it has no globally fixed points in its boundary. This is because it must contain an undistorted free group on two generators and the generators both act by North-South dynamics on the boundary with disjoint fixed points. For a proof of these facts see \cite[Chapter 8]{GH}. The group $G$ is then virtually $\mathbb{Z}$ and so we are done.

\end{proof}

\subsection{Proof of Theorem \ref{thm:conv}}

We will prove Theorem~\ref{thm:conv} by proving the easier to state theorem below.

\begin{theorem} \label{nsdynamics} Let $\gamma^+$ and $\gamma^-$ be points in the contracting boundary. If there is a sequence of isometries $g_i$ such that $g_ix \to \gamma^+$ and $g_i^{-1}x \to \gamma^-$ then for any compact set $K$ in $\partial_cX - \{\gamma^-\}$ and any open neighborhood, $U \subseteq \partial_cX$, of $\gamma^+$, $g_i(K) \subset U$ for large enough $i$.

\end{theorem}

We can loosen the hypothesis that the $g_i^{-1}$ converge to $\gamma^-$ to obtain Theorem~\ref{thm:conv} from Theorem~\ref{nsdynamics}. By a result of Ballman--Buyalo \cite{BB08} if $g_ix \to \gamma^+$ then (passing to a subsequence if necessary) the inverses converge to something in the boundary, lets call it $\gamma^-$. Because the contracting constants of the geodesic $[x,g_i^{-1}x]$ (by Definition~\ref{sequenceconvergence}) are uniformly bounded above by some uniform constant $B$ you can bound the contracting constant of every finite subinterval of $\gamma^-$ by $B+1$ (and in fact by $B$ with a little more work). Thus $\gamma^-$ is contracting as well.

Because open sets in $\partial_cX$ can be much finer than in the visual boundary it is not \textit{a priori} obvious that there will be any form of North-South dynamics on the contracting boundary. The important observation is that all open sets around $\gamma^+$ have a ``B-contracting core'' which contains the set of all B-contracting elements which are nearby to $\gamma^+$ in the visual topology. Because the action by $g_i$ coarsely preserves the contracting constants in $K$, (and because they are already bounded) you can push the set $K$ into the "core" of $U$ with the dynamics of the visual boundary and establish that it is in fact a subset of $U$.

Note: I think this is \textit{not} enough to use the ping-pong lemma because compact sets and neighborhoods aren't compliments of each other like they are with the visual topology. This makes me suspect that there is a decent chance this applies to the Morse boundary (where the ping-pong lemma fails in general see \cite{F15}). Because of this I include a proof of a known dynamics result (Lemma~\ref{NScore}) on the visual boundary of a CAT(0) space which I believe will be amenable to generalization onto the Morse boundary.

The proof will be broken up into two lemmas in order to simplify the discussion. 

For the following we will assume that $X$ is a proper CAT$(0)$ space with non-empty contracting boundary and a group of isometries $G$ acting geometrically. 

\begin{lemma} \label{thickopensets} Let $V$ be an open set in the contracting boundary containing a point $\gamma$, then for each positive constant $B$ there is an $r$ and an $\epsilon$, depending only on, $B$, $\gamma$ and $V$ such that $\partial_c^BX_x \cap U_{\gamma}(r,\epsilon) \subset V$.

\end{lemma}

\begin{proof} We can do this by contradiction. Assume that for some $B$ no such $r$ and $\epsilon$ existed. Then for each $n \in \mathbb{N}$ we could find an element of $\partial_c^BX_x \cap U_{\gamma}(n,1)$ which is not in $V$. Thus we have a sequence of geodesics $\eta_n$ such that $\eta_i \in U_{\gamma}(n,1)$ for all $i \geq n$ which is no more than $B$-contracting. Because this is precisely the condition for convergence of a sequence in the contracting boundary laid out in Lemma~\ref{convergenceincontractingboundary} we have that $\eta_n \to \gamma$ but that the $\eta_n$ are not in $V$. Because $V$ is a neighborhood of $\gamma$ this is a contradiction. $\Rightarrow\Leftarrow$

\end{proof}

The following lemma is a direct consequence of the $\pi$-convergence due to Popasolgu and Swenson in \cite{PS}. This lemma should be generalizable to the Morse boundary so we will provide a different proof which does not rely on the Tits-metric and so is likely easier to generalize.

\begin{lemma} \label{NScore}Let $\gamma^+, \gamma^-$ be elements in $\partial_cX$ and $g_i$ be a sequence of group elements such that $g_ix \to \gamma^+$ and $g_i^{-1}x \to \gamma^-$ in $\overline{X}_c$. For any neighborhoods of $\gamma^-$ and $\gamma^+$ in $\partial X$ of the form $U_{\gamma^-}(s,\varepsilon)$ and $U_{\gamma^+}(r,\epsilon)$, there is an $N$ such that for all points $\alpha$ in the set $\partial_cX - U_{\gamma^-}(s,\varepsilon)$ we have $g_i(\alpha) \subset U_{\gamma^+}(r,\epsilon)$ for all $i \geq N$.

\end{lemma}

\begin{proof} 

\begin{figure}[h]
\begin{tikzpicture}

\draw[<->] (0,0) -- (12,0) node[pos=0,anchor=east]{$\gamma^-$} node[pos=0.95,anchor=north]{$x$} node[pos=0.85,circle,fill,scale=0.3](x){} node[pos=0.12,fill,scale=0.3,circle](Pgi){} node[pos=0.18,anchor = north east]{$\gamma^-(p_{g_i^{-1}x})$} node[pos=0.57,fill,scale=0.3,circle](PA){} node[pos=0.57,anchor = north]{$\gamma^-(p_\alpha)$};

\draw[->] (x) ..controls (7.0,0.2).. (6.5,5.5) node (A) {} node[anchor = south east]{$\alpha$};

\draw (x) ..controls (1.4,0.3).. (1,4.6) node [pos = 0.17,circle,fill,scale = 0.3](pi){} node [circle,fill,scale = 0.3](gix){} node [anchor = south east]{$g_i^{-1}x$};

\node at(pi) [anchor=south east]{$c_i(p_i)$};

\draw[help lines] (gix) -- (Pgi);
\draw[help lines] (PA) -- ($(A.south)-(0.15,0)$);

\node at ($(Pgi)+(1.5,1.1)$) (j1){};
\draw (gix) ..controls ($(Pgi)+(-0.1,1.2)$) and ($(Pgi)+(0.5,1.2)$) .. (j1);

\draw[<-] ($(A.south)-(0.35,0)$) ..controls ($(PA)-(0.3,-1.1)$) and ($(PA)-(0.5,-1.1)$) .. ($(PA)-(1.2,-1.1)$) node [pos = 0.6,scale=0.3,fill, circle](yi){} node [scale=0] (j2){};

\node at (yi) [anchor = south east]{$b_i(y_i)$};

\draw (j1.west) -- (j2.east);

\draw[help lines] (PA) -- (yi);

\end{tikzpicture}
\caption{$p_i$ is uniformly bounded}
\label{fig:boundingpi}
\end{figure}

For the sake of simplicity we can assume that the base point $x$ is on a geodesic from $\gamma^-$ to $\gamma^+$. Through an abuse of notation we will conflate the representatives of $\gamma^-$ and $\gamma^+$ starting at $x$ with the elements $\gamma^-$ and $\gamma^+$. Let us denote the geodesic $[x,g_i^{-1}x]$ by $c_i$. Let $a$ denote the parametrized geodesic from $x$ to $\alpha$ and $b_i$ be the geodesic from $g_i^{-1}x$ to $\alpha$. For the following argument refer to Figure~\ref{fig:boundingpi}.

Denote by $\gamma^{-}(p_\alpha)$ the projection of $\alpha$ onto $\gamma^-$. This exists provided that $\gamma^+ \neq \alpha$, in the case where such a projection is unbounded the geodesic $b_i$ is asymptotic to $\gamma^+$ and in place of $\gamma^-(p_\alpha)$ a point sufficiently far along $\gamma^+$ will suffice since $b_i$ is one leg of a slim ideal triangle. Because $\alpha \in \partial_cX - U_{\gamma^-}(s,\varepsilon)$, there is a uniform bound on $|p_\alpha|$ which depends only on $\gamma^+, \gamma^-, s,$ and $ \varepsilon$. Similarly, denote the projection of $g_i^{-1}x$ onto $\gamma$ by $\gamma^{-}(p_{g_i^{-1}x})$. 

Because the $c_i$ converge to $\gamma^-$, in $\bar{X_c}$, they are uniformly contracting. Thus the triangle given by $\gamma^-$ , $g^{-1}x$, $x$ is $\delta$ slim for some $\delta$ only depending on the bound of the contracting constants of the $c_i$ . A standard argument shows that given an $M'$, for all sufficiently large $i$, we have the inequality $|p_{g^{-1}x} -  p_{\alpha}| > M'$. Choosing the $M'$ from Lemma~\ref{betterthinrectangles} gives us that there is a $y_i$ such that $d(\gamma^-(p_\alpha),b_i(y_i)) < M'$. For large enough $i$ we also have a point on $c_i$, say $c_i(p_i)$, so that  $d(c_i(p_i),\gamma^-(p_\alpha)) <1$. This gives us a bound on the distance $d(c_i(p_i),b_i(y_i)) \leq M'+1$.

Note that the length of $[g_i^{-1}x,b_i(y_i)]$ is no shorter than $|p_{g_i^{-1}x} - p_\alpha| - 2M'$, and so we can make this length larger than $\frac{2(M'+1)r}{\epsilon}$ by picking yet larger $i$.

Shifting the picture by applying the isometry $g_i$ gives us Figure~\ref{fig:convergenceofgibi}. The previous estimation was done to arrange it so that $g_ia$ is in $U_{g_ic_i}(r,\frac{\epsilon}{2})$. Because the $g_ic_i$ converge to $\gamma^+$ we can assume that $g_ic_i$ is in $U_{\gamma^+}(r,\frac{\epsilon}{2})$. By setting $N$ to be the largest of the previous $i$'s we get that $g_ia \in U_{\gamma^+}(r,\epsilon)$. Note that none of the previous estimates depend on $\alpha$, (including the bound on $|p_\alpha|$).

\begin{figure}
\begin{tikzpicture}

\draw[<->] (0,0) -- (12,0) node[pos=1,anchor=west]{$\gamma^+$} node[pos=0.05,anchor=north]{$x$} node[pos=0.05,circle,fill,scale=0.3](x){} node[pos=0.66,circle,fill,scale=0.3](R){};

\node at (R) [anchor = north]{$\gamma^+(r)$};

\draw[->] (x) ..controls (8.8,0.7).. (8.4,3.4) node [pos = 0.47,circle,fill,scale=0.3](r'){} node[pos=0.72,circle,fill,scale=0.3](w1){} node[anchor = south east]{$g_i\alpha$};

\node at (w1) [anchor=south east]{$g_i(b_i(y_i))$};

\draw (x) -- (11,0.7) node[anchor=west]{$g_ix$} node[circle,fill,scale=0.3](gix){} node[pos=0.83,fill,scale=0.3,circle](P){} node[pos=0.8,anchor = north west]{$g_i(c_i(p_i))$} node [pos=0.69,scale=0.3,circle,fill](r){};

\draw[help lines] (w1)--(P) node [pos=0.5,anchor=south west,color=black]{$M'+1$};

\draw[help lines] (r) -- (R) node [pos = 0.5,anchor = west,color = black]{$\frac{\epsilon}{2}$};

\draw[help lines] (r') -- (r)node [pos = 0.3,anchor = west,color = black]{$\frac{\epsilon}{2}$};

\end{tikzpicture}

\caption{Convergence of the $g_i(b_i)$.}
\label{fig:convergenceofgibi}
\end{figure}
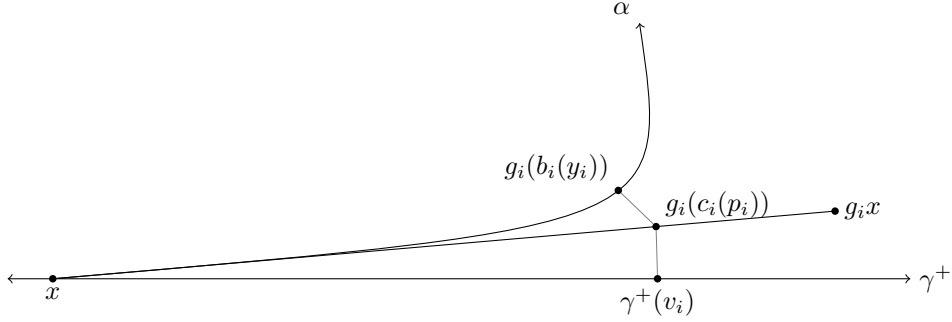

\end{proof}

\begin{proof}[Proof of Theorem \ref{thm:conv}]

We may assume that $x$ is on the bi-infinite geodesic $\gamma$ from $\gamma^-$ to $\gamma^+$. Let $K$ be a compact set in $\partial_cX - \gamma^-$ and $U$ an open set containing $\gamma^+$ in $\partial_cX$. By Lemma~\ref{compactsets} there is a uniform $A$ such that all elements $\alpha$ in $K$ (with basepoint $x$) are no more than $A$-contracting.

Because $[x,g_ix] \to \gamma^+$ by Lemma~\ref{sequenceconvergence} they are no more than $B$-contracting where $B$ depends only on the sequence of $g_i$. For all $i$ and all $\alpha \in K$ by Lemma~\ref{lemma:concatenationsarecontracting} the geodesic $[x,g_i\alpha)$ is $\Phi_{\ref{lemma:concatenationsarecontracting}}(A,B)$-contracting because the geodesics $[x,g_ix]$ is $B$-contracting and $[g_i x,g_i\alpha)$ is $A$-contracting. For notational convenience set $C = \Phi_{\ref{lemma:concatenationsarecontracting}}(A,B)$.

By Lemma~\ref{thickopensets} for the contracting constant $C$ there is an $r$ and an $\epsilon$ such that $\partial_c^CX_x\cap U_{\gamma^+}(r,\epsilon) \subset U$. Because $K$ is a compact set in the contracting boundary by Lemma~\ref{compactsets} it is also a compact set in $\partial X - \{\gamma^-\}$, so there is an $s$ and an $\varepsilon$ such that $K \subseteq \partial X_x - U_{\gamma^-}(s,\varepsilon)$. So applying Lemma~\ref{NScore}, for large enough $i$ we have that $g_i(K) \subset U_{\gamma^+}(r,\epsilon)$, but we already know that $g_i(K) \subset \partial_c^CX_x$ and so $g_i(K) \subset U$.

\end{proof}

\section{A characterization of \texorpdfstring{$\delta$}{delta}-hyperbolicity}
\label{sec:metrizability}

One of the ways in which the behavior of the contracting boundary diverges from that of the Gromov boundary is in its local topology. The Gromov boundary comes equipped with a family of visual metrics that induce the same topology on the boundary, making it a compact, complete, metric space. For the contracting boundary this happens only in the rarest of circumstances. It is quite easy to cook up examples of spaces which have non-metrizable contracting boundary. The following is one such example.

Consider again our favorite RAAG, $A_\Gamma = \langle a, b, c \; | \; [b,c]\rangle$, along with the universal cover of its Salvetti complex, $\tilde{X}$. The infinite word $w = aaaa \cdots$ corresponds to a 0-contracting geodesic in $\tilde{X}$ which starts at some lift of the natural base point $x$ in $X$ (see figure \ref{fig:salvetti}). If we let $w_i^j = a^ib^jaaa\cdots$, this corresponds to an infinite geodesic starting at the lift of $x$ which is exactly $j$-contracting (i.e. it is not $B$-contracting for any $B< j$). It is clear that for each fixed $j$ the sequences $\{w_i^j\}_{i \in \mathbb{N}}$ converge to $w$ in the contracting boundary. Now if we construct a new sequence by picking an $i$ for each $j$, i.e. we choose a function $f: \mathbb{N} \to \mathbb{N}$, then regardless of our choice of $f$ the new sequence $\{w^j_{f(j)}\}_{j\in\mathbb{N}}$ will never converge to $w$. This is because the set $\{w^j_{f(j)}\}_{j\in\mathbb{N}}$ is closed in $\partial_c\tilde{X}$, as its intersection with each component, $\partial_c^D\tilde{X}_x$, is finite and therefore closed.

It is a general fact for all first countable spaces that if you have a countable collection of sequences which all converge to the same point, it is always possible to pick a `diagonal' sequence which also converges. i.e. if we have $\{x_i^j\}$ such that $\displaystyle \lim_{i} x_i^j = x$ there is always some function $f : \mathbb{N} \to \mathbb{N}$ such that $\displaystyle \lim_j x^j_{f(j)} = x$. The proof of this is an elementary exercise in point-set topology. Because this is impossible in the above example we can see that the contracting boundary of $\tilde{X}$ cannot be metrizable.

Of course, for some CAT(0) spaces, the contracting boundary is metrizable, any CAT(-1) spaces for instance. It turns out that this is completely generic, the metrizability of the contracting boundary completely characterizes $\delta$-hyperbolicity of cocompact CAT(0) spaces. 

\begin{theorem} \label{thm:characterizationofhyperbolicity} Assume that there is a group $G$ acting geometrically on a complete proper CAT(0) space, $X$, with $|\partial_cX| > 2$ then the following are equivalent:

\smallskip

\begin{sublemmas}

\item $X$ is $\delta$-hyperbolic.

\smallskip
\item The contracting constants are bounded i.e. $\partial_cX_{x_0} \subseteq \partial_c^D X_{x_0}$ for some $D$. 

\smallskip
\item The map $Id: X \to X$ induces a homeomorphism $\partial X \cong \partial_cX$.

\smallskip
\item $\partial X \subseteq \partial_cX$ i.e. as sets the visual boundary and the contracting boundary are the same.

\smallskip
\item $\partial_cX$ is compact.

\smallskip
\item $\partial_cX$ is locally compact.

\smallskip
\item $\partial_cX$ is first-countable, and in fact metrizable.

\end{sublemmas}
\end{theorem}

\bigskip

In order to prove these equivalences we need a bit more fine control over how the contracting constants change under the group action. When there is a rank-one isometry you can say precisely how the contracting constants are changing as you act on a contracting ray. We will make that more precise below, but first we need some notation.

\emph{Notation:} If $b$ is a $B$-contracting geodesic in some CAT(0) space $X$ then I will denote the minimum of all contracting constants $\tilde{B} := \min\{B \; | \; b \mbox{ is  $B$-contracting }\}$.

\begin{lemma} \label{actionpreservesconstants} Let $g$ be a rank one isometry of a CAT(0) space $X$ whose axis, $a$, is $A$-contracting. If $b$ is a $B$-contracting geodesic with $b(0) = a(0)= p$, then $k_n = [p,g^nb(\infty))$ will be a $K$-contracting geodesic such that 

		$$ \Psi(\tilde{B},A) \leq  K  \leq  \Phi(A,B) $$
\smallskip

\noindent where $\displaystyle \Psi(\tilde{B},A) = \frac{\tilde{B} - 16A -77\delta_A - 38}{16} - 3$ and $\Phi(A,B)$ is as in Lemma~\ref{lemma:concatenationsarecontracting}.

\end{lemma}

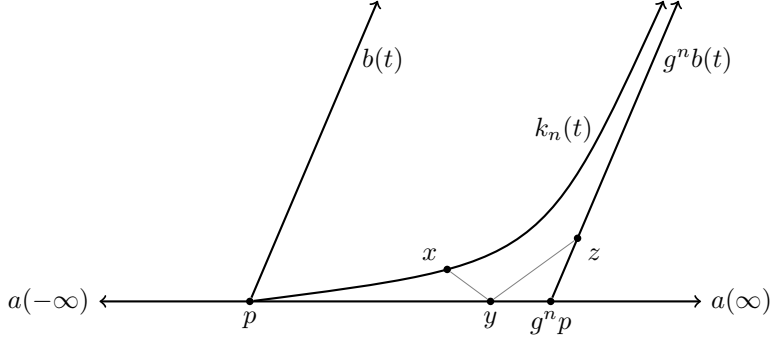
\begin{figure}
\begin{tikzpicture}[thick]
	
	\node[anchor=north] at (0,0){$g^np$} node[fill,circle,scale=0.3] at (0,0) (gnp){};
	\node[anchor=north] at (-4,0){$p$} node[fill,circle,scale=0.3] at (-4,0) (p){};
	\draw[<->] (-6,0) -- (2,0) node[pos=1,anchor= west]{$a(\infty)$} node[pos=0,anchor= east]{$a(-\infty)$} node[fill,circle,scale=0.3,pos=0.65](y){} node[pos=0.65,anchor=north]{$y$};
	\draw[->] (p) ..controls (-0.1,0.5).. (1.5,4) node[circle,fill,scale=0.3,pos=0.3](x){} node[pos=0.3,anchor=south east]{$x$} node[pos=0.8,anchor=east]{$k_n$};
	\draw[->] (gnp) -- (1.7,4)
node[circle,fill,scale=0.3,pos=0.2](z){} node[pos=0.2,anchor=north west]{$z$} node[pos=0.8,anchor=west]{$g^nb$};
    \draw[->] (p) -- (-2.3,4)node[pos=0.8,anchor=west]{$b$};
	\draw[help lines] (x) -- (y); %node[midway,anchor=east,black]{$\delta_{A}$};
	\draw[help lines] (y) -- (z); %node[pos=0.5,anchor=  east,black]{$\delta_A$};

\end{tikzpicture}
\caption{Periodic isometries coarsely fix contracting constants} \label{fig:fixedcontractingconstant}
\end{figure}

\begin{proof} Consider the geodesics $g^nb$ and $k_n$. By Lemma~\ref{lemma:concatenationsarecontracting}, because $a$ is $A$ contracting and $b$ is $B$ contracting, the geodesic $k_n$ is at most $\Phi_{\ref{lemma:concatenationsarecontracting}}(A,B)$-contracting. 

Assume for the sake of contradiction that $k_n$ is $K$-contracting with $$\displaystyle K < \frac{\tilde{B} - 16A -77\delta_A - 38}{16} - 3$$

In particular this gives us that $K+3 < \frac{\tilde{B} - 16A -77\delta_A - 38}{16}$

\bigskip
Because $a$ is $\delta_A$-slim and by replacing $\delta_A$ with $2\delta_A$ if necessary, then there is an $x \in k_n$ a $y \in a$ and a $z \in g^nb$ with $d(x,y) \leq \delta_A$ and $d(y,z) \leq \delta_A$. Now because $k_n$ is $K$ contracting the subsegment $[x,k_n(\infty)]$ is $K+3$-contracting. The geodesic $[z,g^nb(\infty)]$ is within the $2\delta_A$ neighborhood of $[x,k_n(\infty)]$ so Lemma~\ref{lemma:closetocontracting} gives us an explicit upper bound on the contracting constant for $[z,g^nb(\infty)]$. In particular we know that it is at worst $\Phi_{\ref{lemma:closetocontracting}}(K+3,2\delta_A,2\delta_A)$-contracting where

$$\Phi_{\ref{lemma:closetocontracting}}(K+3,2\delta_A,2\delta_A) = 16(K+3)+70\delta_A+10$$

\bigskip

Similarly, we can see that $[g^np,z]$ is $\Phi_{\ref{lemma:closetocontracting}}(A+3,0,\delta_A) = (16A + 7\delta_A + 10)$-contracting. Now $g^nb$ is the concatatination of $[g^np,z]$ and $[z,g^nb(\infty))$ and so we get that it is at most $B' = \Phi_{\ref{lemma:closetocontracting}}(K+3,2\delta_A,2\delta_A)+\Phi_{\ref{lemma:closetocontracting}}(A+3,0,2\delta_A)$-contracting. 

\bigskip
Working everything out, the assumption that we made gives us the following inequality:

$$\begin{array}{ccl}
B'   & = &       16(K+3) + 16A + 77\delta_A + 38 \\
    &   &                                       \\
    & < &    16\left(\frac{\tilde{B} - 16A -77\delta_A - 38}{16}\right) + 16A + 77\delta_A + 38 \\
    &   &                                       \\
    & = &       \tilde{B} \end{array}$$

\noindent But then $g^nb$ is $B$-contracting with $B < \tilde{B}$ which is a contradiction $\Rightarrow\Leftarrow$.

\bigskip
\noindent So we know that $k_n$ is $K$-contracting where $K \geq \Psi(\tilde{B},A)$

\end{proof}

Corollary \ref{therearerankoneaxes} provides us with a rank-one axis whenever the contracting boundary is non-empty and so Lemma~\ref{actionpreservesconstants} gives us fine tuned control over the contracting constants under the action of that rank-one isometry.

\begin{remark} \label{unboundedcontractingconstants}Suppose we have a sequence of contracting geodesics $\{k_n\}$ and another \emph{non}-contracting geodesic $b$ all with the same base point. If the end points $k_n(\infty)$ converge to $b(\infty)$ in the visual boundary, then the contracting constants for $k_n$ are unbounded.

\end{remark}

We now have all of the ingredients we needed in order to prove the main theorem.

\begin{proof}[Proof of Theorem \ref{thm:characterizationofhyperbolicity}] I will first prove the equivalence of $(i)$ through $(iv)$. The equivalence of $(v), (vi)$, and $(vii)$ with the others will then be easier to show. 

\bigskip
$(i) \Longrightarrow (ii)$ The slim triangle condition for a $\delta$-hyperbolic space is easily seen to imply the slim geodesic condition that we have been using, for an explicit proof see \cite{CS13}. Because every geodesic is uniformly $\delta$-slim by the hyperbolicity condition they all have uniform contracting constants by Lemma~\ref{lemma:slimgeodesics}.

\bigskip

$(ii) \Longrightarrow (iii)$ Note that because the contracting constants are bounded, the directed system stabilizes i.e. the collection of contracting geodesics has the subspace topology induced from the visual boundary. Thus, if we can prove that every infinite ray is contracting we would be done, since the contracting boundary will then have the same topology as the entire visual boundary. 

Let $b$ be some geodesic ray in $X$ and pick any $a \in \partial_cX$, by Corollary~\ref{por:orbitsarevisuallydense} there is a sequence of $\{g_i\}$ such that $g_ia \to b$. If $b$ is not contracting then by Remark~\ref{unboundedcontractingconstants} the contracting constants of the representatives of $g_ia$ that start at $x = a(0)$ are growing without bound, contradiction $\Rightarrow\Leftarrow$.

\bigskip

$(iii) \Longrightarrow (iv)$ This is obvious.

\bigskip

$(iv) \Longrightarrow (i)$ This follows from The Flat Plane Theorem. i.e. since $\partial X \subseteq \partial_cX$ we have that every geodesic in $X$ is contracting, thus there are no non-contracting geodesics. In particular this implies that there are no Euclidean planes embedded in $X$, but by the above theorem this implies that $X$ is $\delta$-hyperbolic.

\bigskip
$(iii) \implies (v)$ For a proper CAT(0) space $\partial X$ is compact.

\bigskip
$(v) \implies (vi)$ This is trivial.

\bigskip
$(vi) \implies (ii)$ Assume $(ii)$ is false, then I will show that $\partial_cX$ is not locally compact. 

Let $\alpha$ be some element of $\partial_cX$ and let $U$ be an arbitrary neighborhood of $\alpha$. By Corollary~\ref{therearerankoneaxes} there is some rank-one isometry, $g$, and by Theorem~\ref{thm:orbitsaredense} we may assume that the forward end point of $g$ is in the interior of $U$. Let $a$ be an axis of $g$ and let $A$ be its contracting constant.

By assumption there is a subset $B = \{b_i\}$ such that the minimal contracting constant of each $b_i$ is bounded below by $16(i + 16A + 77\delta_A+38)$. Applying Theorem~\ref{thm:orbitsaredense} with the $g^n$ as the $g_i$ and switching the roles of $a$ and $b_i$ we can see that $g^nb_i(\infty)$ converges to $a(\infty)$. In particular, for each $i$ there is an $n$, say $n_i$, such that $g^{n_i}b_i(\infty)$ is in $U$. Let the geodesics $[a(0),g^{n_i}b_i(\infty))$ be denoted by $c_i$. Applying Lemma~\ref{actionpreservesconstants} we get that the $c_i$ are at least $i$ contracting.

 Let the collection $\mathcal{C} = \{C_i\}$ where $\displaystyle C_i = \{c_j\}_{j \geq i}$. Each set $C_i$ is a closed subset of $U$ and $\bigcap C_i = \varnothing$. The collection $\{U\setminus C_i\}$ will then be an open cover of $U$ with no finite refinement and so U is not compact. Since $\alpha$ and $U$ were arbitrary $\partial_cX$ is not locally compact.

\bigskip
$(i) + (iii) \Longrightarrow (vii)$ The Gromov boundary of a $\delta$-hyperbolic group is metrizable and since the contracting boundary is homeomorphic to the visual boundary we are done.

\bigskip
$(vii) \Longrightarrow (ii)$ Assume that (ii) is false, that there is no upper bound on the contracting constants of the contracting boundary. We will show that the contracting boundary is not first countable (and thus not metrizable). 

\bigskip
As with the example in the introduction to this Section it is enough to exhibit a collection $\{\alpha_i^j\}$ and an $\alpha$ in the contracting boundary such that for each $j$, $\alpha_i^j \to \alpha$ as $i \to \infty$, but the $\alpha_i^j$ are at best $j$ contracting. In particular this means that for any function $f : \mathbb{N} \to \mathbb{N}$ the sequence $\alpha_{f(j)}^{j}$ will not converge to $\alpha$. This is because the intersection of $\{\alpha_{f(j)}^j\}$ with $\partial_c^DX_{x_0}$ will always be finite and thus the set $\{\alpha_{f(j)}^j\}$ is closed. We've already seen that the existence of such a sequence contradicts first countability. 

\bigskip
The construction of the $\alpha_i^j$'s aren't particularly hard in light of Lemma~\ref{actionpreservesconstants}. Since $\partial_cX$ is non-empty we have by Proposition~\ref{therearerankoneaxes} a rank-one isometry $g$ with axis $a$. Now since $a$ is rank-one it has a contracting constant $A$ and is $\delta_A$ slim. We are assuming that there is no upper bound on the contracting constants for geodesics so pick a geodesic $b^j$ with a minimal contracting constant $\tilde{B_j}$ of at least $16j + 16A + 77\delta_A + 38$. By Lemma~\ref{actionpreservesconstants} the geodesics $k_i^j = [b(0),g^ib^j(\infty)]$ will be $K$-contracting where $j \leq \Psi(\tilde{B_j},A)\leq K \leq \Phi_{\ref{actionpreservesconstants}}(A,\tilde{B}_j)$.

So we have our collection of points in the contracting boundary $\{k_i^j(\infty)\}$. For each $j$ the geodesics $k_i^j$ have a fixed upper bound on their contracting constants. To get convergence in the visual boundary recall that a rank one isometry acts by North-South dynamics on the visual boundary \cite{H09}. Thus $\displaystyle \lim_i k_i^j(\infty) = a(\infty)$ for each $j$ in $\partial_cX$ and the contracting constants are bounded below by $j$. This gives us that $\partial_cX$ cannot be first countable.

\end{proof}

\bibliographystyle{utphys}
\bibliography{contracting_boundary}
%\nocite{*}

\end{document}